\documentclass{siamart0216}
\usepackage{amsfonts}
\usepackage{graphicx}
\usepackage{epstopdf}
\usepackage{calligra}
\usepackage{algpseudocode}

\usepackage{amsfonts,amsmath,amssymb}
\usepackage{mathrsfs,mathtools,stmaryrd}
\usepackage{enumerate}
\usepackage{xspace} 
\usepackage{esint}
\usepackage{graphicx}
\DeclareMathAlphabet{\mathpzc}{OT1}{pzc}{m}{it}

\DeclareMathOperator{\diag}{diag}
\DeclareMathOperator{\Proj}{Proj}

 \newsiamremark{remark}{Remark}

\newcommand{\calL}{{\mathcal L}}

\newcommand{\R}{\mathbb{R}}
\newcommand{\usf}{\mathsf{u}}
\newcommand{\asf}{\mathsf{a}}
\newcommand{\bsf}{\mathsf{b}}
\newcommand{\usfd}{\mathsf{u}_{\textup{d}}}
\newcommand{\zsf}{\mathsf{z}}
\newcommand{\calLs}{\mathcal{L}^s}
\newcommand{\diff}{\, \mbox{\rm d}}
\newcommand{\ie}{i.e.,\@\xspace}

\newcommand{\Hs}{\mathbb{H}^s(\Omega)}
\newcommand{\Ws}{\mathbb{H}^{1-s}(\Omega)}

\newcommand{\GL}{{\textup{\textsf{GL}}}}

\newcommand{\Nin}{\,{\mbox{\,\raisebox{6.0pt} {\tiny$\circ$} \kern-10.9pt}\N }}
\newcommand{\DIV}{\textrm{div}}
\newcommand{\C}{\mathcal{C}}
\newcommand{\ue}{\mathscr{U}}
\newcommand{\V}{\mathbb{V}}
\newcommand{\U}{\mathbb{U}}
\newcommand{\I}{\mathcal{I}}
\newcommand{\T}{\mathscr{T}}
\newcommand{\Tr}{\mathbb{T}}
\newcommand{\rsf}{\mathsf{r}}
\newcommand{\ope}{\bar{\mathscr{P}}}
\newcommand{\oue}{\bar{\mathscr{U}}}
\newcommand{\orsf}{\bar{\mathsf{r}}}
\newcommand{\opsf}{\bar{\mathsf{p}}}
\newcommand{\ousf}{\bar{\mathsf{u}}}
\newcommand{\ozsf}{\bar{\mathsf{z}}}
\newcommand{\pe}{\mathscr{P}}
\newcommand{\wsf}{\mathsf{w}}
\newcommand{\psf}{\mathsf{p}}
\newcommand{\Zad}{\mathsf{Z}_{\textrm{ad}}}
\newcommand{\HLn}{{\mbox{\,\raisebox{5.1pt} {\tiny$\circ$} \kern-9.3pt}{H}^1_L  }}
\DeclareMathOperator*{\tr}{tr_\Omega}
\newcommand{\HL}{ \mbox{ \raisebox{7.0pt} {\tiny$\circ$} \kern-10.7pt} {H_L^1} }
\newcommand{\Hsd}{\mathbb{H}^{-s}(\Omega)}
\newcommand{\Y}{\mathpzc{Y}}

\newcommand{\TheTitle}{Sparse optimal control for fractional diffusion}
\newcommand{\ShortTitle}{Sparse control}
\newcommand{\TheAuthors}{E.~Ot\'arola, A.J.~Salgado}

\headers{\ShortTitle}{\TheAuthors}

\title{{\TheTitle}\thanks{EO has been supported in part by CONICYT through FONDECYT project 3160201. AJS has been supported in part by NSF grant DMS-1418784.}}

\author{
  Enrique Ot\'arola\thanks{Departamento de Matem\'atica, Universidad T\'ecnica Federico Santa Mar\'ia, Valpara\'iso, Chile.
    (\email{enrique.otarola@usm.cl}, \url{http://eotarola.mat.utfsm.cl/}).}
  \and
  Abner J.~Salgado\thanks{Department of Mathematics, University of Tennessee, Knoxville, TN 37996, USA.
    (\email{asalgad1@utk.edu}, \url{http://www.math.utk.edu/\string~abnersg})}
}

\ifpdf
\hypersetup{
  pdftitle={\TheTitle},
  pdfauthor={\TheAuthors}
}
\fi

\date{Submitted \today.}

\begin{document}

\maketitle

\begin{abstract}
We consider an optimal control problem that entails the minimization of a nondifferentiable cost functional, fractional diffusion as state equation and constraints on the control variable.  We provide existence, uniqueness and regularity results together with first order optimality conditions. In order to propose a solution technique, we realize fractional diffusion as the Dirichlet-to-Neumann map for a nonuniformly elliptic operator and consider an equivalent optimal control problem with a nonuniformly elliptic equation as state equation. The rapid decay of the solution to this problem suggests a truncation that is suitable for numerical approximation. We propose a fully discrete scheme: piecewise constant functions for the control variable and first--degree tensor product finite elements for the state variable. We derive a priori error estimates for the control and state variables which are quasi--optimal with respect to degrees of freedom.
\end{abstract}

\begin{keywords}
optimal control problem, nondifferentiable objective, sparse controls, fractional diffusion, weighted Sobolev spaces, finite elements, stability, anisotropic estimates.
\end{keywords}

\begin{AMS}
26A33,    
35J70,    
49K20,    
49M25,    
65M12,    
65M15,    
65M60.    
\end{AMS}

\section{Introduction}
\label{sec:intro}
In this work we shall be interested in the design and analysis of a numerical technique to approximate the solution to a nondifferentiable optimal control problem involving the fractional powers of a uniformly elliptic second order operator; control constraints are also considered. To make matters precise, let $\Omega$ be an open and bounded polytopal domain of $\R^n$ with $n \geq 1$. Given $s \in (0,1)$ and a desired state $\usfd: \Omega \rightarrow \R$, we define the nondifferentiable cost functional
\begin{equation}
\label{eq:J}
J(\usf,\zsf) = \frac{1}{2} \|\usf-\usfd\|_{L^2(\Omega)}^2 + \frac{\sigma}{2} \|\zsf \|_{L^2(\Omega)}^2 + \nu \|\zsf\|_{L^1(\Omega)},
\end{equation}
where $\sigma$ and $\nu$ are positive parameters. We shall thus be concerned with the following nondifferentiable optimal control problem: Find
\begin{equation}
\label{eq:min}
\min J(\usf,\zsf)
\end{equation}
subject to the \emph{fractional state equation}
\begin{equation}
\label{eq:state_equation}
\calLs \usf = \zsf \textrm{ in } \Omega,
\end{equation}
and the \emph{control constraints}
\begin{equation}
\label{eq:cc}
\asf \leq \zsf(x') \leq \bsf \quad \textrm{a.e.} \quad x' \in \Omega.
\end{equation}
The operator $\calLs$, with $s \in (0,1)$, is a fractional power of the second order, symmetric, and uniformly elliptic operator
\begin{equation}
\label{eq:L}
 \mathcal{L} w= - \DIV_{x'} (A(x') \nabla_{x'} w) + c(x') w,
\end{equation}
supplemented with homogeneous Dirichlet boundary conditions; $0 \leq c \in L^{\infty}(\Omega)$ and $A \in C^{0,1}(\Omega,\GL(n,\R))$ is symmetric and positive definite. The control bounds $\asf, \bsf \in \R$ and, since we are interested in the nondifferentiable scenario, we assume that $\asf < 0 < \bsf$ \cite[Remark 2.1]{CHW:12}.

The design of numerical techniques for the optimal control problem \eqref{eq:min}--\eqref{eq:cc} is mainly motivated by the following considerations:
\begin{enumerate}[$\bullet$]
 \item Fractional diffusion has recently become of great interest in the applied sciences and engineering: practitioners  claim that it seems to better describe many processes. For instance, mechanics \cite{atanackovic2014fractional}, biophysics \cite{bio}, turbulence \cite{wow}, image processing \cite{GH:14}, nonlocal electrostatics \cite{ICH} and finance \cite{MR2064019}. It is then natural the interest in efficient approximation schemes for problems that arise in these areas and their control.
 
 \item The objective functional $J$ contains an $L^1(\Omega)$--control cost term that leads to sparsely supported optimal controls; a desirable feature, for instance, in the optimal placement of discrete actuators \cite{MR2556849}. This term is also relevant in settings where the control cost is a linear function of its magnitude \cite{MR2283487}.
\end{enumerate}

One of the main difficulties in the study and discretization of the state equation \eqref{eq:state_equation} is the nonlocality of the fractional operator $\calLs$ \cite{CS:07,CDDS:11,ST:10}. A possible approach to this issue is given by a result of Caffarelli and Silvestre in $\mathbb{R}^n$ \cite{CS:07} and its extensions to bounded domains \cite{CDDS:11,ST:10}: Fractional powers of $\mathcal{L}$ can be realized as an operator that maps a Dirichlet boundary condition to a Neumann condition via an extension problem on the semi--infinite cylinder $\C = \Omega \times (0,\infty)$. Therefore, we shall use the Caffarelli--Silvestre extension to rewrite the fractional state equation \eqref{eq:state_equation}
as follows:
\begin{equation}
\label{eq:harmonic_extension}
  -\DIV \left( y^{\alpha} \mathbf{A} \nabla \ue \right) + y^{\alpha} c\ue = 0  \textrm{ in } \C, \quad
  \ue = 0 \text{ on } \partial_L \C, \quad
  \frac{ \partial \ue }{\partial \nu^\alpha} = d_s \zsf  \text{ on } \Omega \times \{0\},
\end{equation}
where $\partial_L \C= \partial \Omega \times [0,\infty)$ is the lateral boundary of $\C$, $\alpha = 1-2s \in (-1,1)$, $d_s = 2^{\alpha}\Gamma(1-s)/\Gamma(s)$ and 
the conormal exterior derivative of $\ue$ at $\Omega \times \{ 0 \}$ is
\begin{equation}
\label{def:lf}
\frac{\partial \ue}{\partial \nu^\alpha} = -\lim_{y \rightarrow 0^+} y^\alpha \ue_y;
\end{equation}
the limit being understood in the distributional sense \cite{CS:07,CDDS:11,ST:10}. Finally, the matrix $\mathbf{A} \in C^{0,1}(\C,\GL(n+1,\R))$ is defined by $\mathbf{A}(x',y) =  \diag \{A(x'),1\} $. We will call $y$ the \emph{extended variable} and the dimension $n+1$ in $\R_+^{n+1}$ the \emph{extended dimension} of problem \eqref{eq:harmonic_extension}. As noted in \cite{CS:07,CDDS:11,ST:10}, $\mathcal{L}^s$ and the Dirichlet-to-Neumann operator of \eqref{eq:harmonic_extension} are related by
\[
 d_s \mathcal{L}^s \usf = \partial_{\nu}^{\alpha} \ue \quad \text{in } \Omega \times \{ 0\}.
\]

The analysis of optimal control problems involving a functional that contains an $L^1(\Omega)$--control cost term
has been previously considered in a number of works. The article \cite{MR2556849} appears to be the first to provide an analysis when the state equation is a linear elliptic PDE: the author utilizes a regularization technique that involves an $L^2(\Omega)$--control cost term, analyze optimality conditions, and study the convergence properties of a proposed semismooth Newton method. These results were later extended in \cite{MR2826983}, where the authors obtain rates of convergence with respect to a regularization parameter. Subsequently, in \cite{CHW:12}, the authors consider a semilinear elliptic PDE as state equation and analyze second order optimality conditions. Simultaneously, the numerical analysis based on finite element techniques has also been developed in the literature. We refer the reader to \cite{MR2826983}, where the state equation is a linear elliptic PDE and to \cite{CHW:12again,CHW:12} for extensions to the semilinear case. The common feature in these references, is that, in contrast to \eqref{eq:state_equation}, the state equation is local. To the best of our knowledge, this is the first work addressing the analisys and numerical approximation of \eqref{eq:min}--\eqref{eq:cc}.

The main contribution of this work is the design and analysis of a solution technique for the \emph{fractional optimal control problem} \eqref{eq:min}--\eqref{eq:cc}. We overcome the nonlocality of $\calLs$ by using the Caffarelli--Silvestre extension: we realize the state equation \eqref{eq:state_equation} by \eqref{eq:harmonic_extension}, so that our problem can be equivalently written as: Minimize $J(\ue|_{y=0},\zsf)$ subject to the extended state equation \eqref{eq:harmonic_extension} and the control constraints \eqref{eq:cc}; \emph{the extended optimal control problem}. We thus follow \cite{MR3429730,MR3504977} and propose the following strategy to solve our original control problem \eqref{eq:min}--\eqref{eq:cc}: given a desired state $\usfd$, employ the finite element techniques of \cite{NOS} and solve the equivalent optimal control problem. This yields an optimal control $\zsf: \Omega \rightarrow \R$ and an optimal extended state $\ue: \C \rightarrow \R$. Setting $\usf(x') = \ue(x',0)$ for all $x' \in \Omega$, we obtain the optimal pair $(\usf,\zsf)$ that solves \eqref{eq:min}--\eqref{eq:cc}.

The outline of this paper is as follows. In section~\ref{sec:not_and_pre} we introduce notation, define fractional powers of elliptic operators via spectral theory, introduce the functional framework that is suitable to analyze problems \eqref{eq:state_equation} and \eqref{eq:harmonic_extension} and recall elements from convex analysis. In section~\ref{sec:fractional}, we study the fractional optimal control problem. We derive existence and uniqueness results together with first order necessary and sufficient optimality conditions. In addition, we study the regularity properties of the optimal variables. In section \ref{sec:extended} we analyze the extended optimal control problem. We begin with the numerical analysis for our optimal control problem  in section \ref{sec:truncated}, where we introduce a truncated problem and derive approximation properties of its solution. Section \ref{sec:approximation} is devoted to the design and analysis of a numerical scheme to approximate the solution to the control problem \eqref{eq:min}--\eqref{eq:cc}: we derive a priori error estimates for the optimal control variable and the state. 

\section{Notation and Preliminaries}
\label{sec:not_and_pre}

In this work $\Omega$ is a bounded and open convex polytopal subset of $\R^n$ ($n\geq1$) with boundary $\partial\Omega$. The difficulties inherent to curved boundaries could be handled with the arguments developed in \cite{Otarola_controlp1} but this would only introduce unnecessary complications of a technical nature.

We follow the notation of \cite{MR3429730,NOS} and define the semi--infinite cylinder with base $\Omega$ and its lateral boundary, respectively, by $\C = \Omega \times (0,\infty)$ and $\partial_L \C  = \partial \Omega \times [0,\infty)$. For $\Y>0$, we define the truncated cylinder
$
  \C_\Y = \Omega \times (0,\Y)
$
and 
$\partial_L\C_\Y$ accordingly. 

Throughout this manuscript we will be dealing with objects defined on $\R^n$ and $\R^{n+1}$. It will thus be important to distinguish the extended $(n+1)$--dimension, which will play a special role in the analysis. We denote a vector $x\in \R^{n+1}$ by $x = (x',y)$ with $x' \in \R^n$ and $y\in\R$.

In what follows the relation $A \lesssim B$ means that $A \leq c B$ for a nonessential constant whose value might change at each occurrence.

\subsection{Fractional powers of second order elliptic operators}
\label{sub:fractional_powers}
We proceed to briefly review the spectral definition of the fractional powers of the second order elliptic operator $\calL$, defined in \eqref{eq:L}. To accomplish this task we invoke the spectral theory for $\mathcal{L}$, which yields the existence of a countable collection of eigenpairs $\{(\lambda_k,\varphi_k)\}_{k \in \mathbb N} \subset \R_+ \times H^1_0(\Omega)$ such that
\begin{equation*}
  \label{eigenvalue_problem_L}
    \mathcal{L} \varphi_k = \lambda_k \varphi_k  \text{ in } \Omega,
    \qquad
    \varphi_k = 0 \text{ on } \partial\Omega, \qquad k \in \mathbb N.
\end{equation*}
In addition, $\{\varphi_k\}_{k \in \mathbb N}$ is an orthonormal basis of $L^2(\Omega)$ and an orthogonal basis of $H_0^1(\Omega)$. Fractional powers of  $\mathcal L$, are thus defined by
\begin{equation}
  \label{def:second_frac}
  \mathcal{L}^s w  := \sum_{k=1}^\infty \lambda_k^{s} w_k \varphi_k
  \quad \forall w \in C_0^{\infty}(\Omega), \qquad s \in (0,1),
  \quad
  w_k = \int_{\Omega} w \varphi_k \diff x'.
\end{equation}
Invoking a density argument, the previous definition can be extended to
\begin{equation}
\label{Hs}
  \Hs = \left\{ w = \sum_{k=1}^\infty w_k \varphi_k \in L^2(\Omega): \| w \|^2_{\Hs}:= \sum_{k=1}^\infty \lambda_k^s |w_k|^2 < \infty \right\}.
\end{equation}
This space corresponds to $[L^2(\Omega),H_0^1(\Omega)]_{s}$ \cite[Chapter 1]{Lions}. Consequently, if $s \in (\tfrac12,1)$, $\Hs$ can be characterized by
\[
  \Hs = \left\{ w \in H^s(\Omega): w = 0 \text{ on } \partial\Omega \right\},
\]
and, if $s \in  (0,\tfrac12)$, then $\Hs = H^s(\Omega) = H_0^s(\Omega)$. If $s=\frac{1}{2}$, the space $\mathbb{H}^{\frac{1}{2}}(\Omega)$ corresponds to the so-called \emph{Lions--Magenes} space \cite[Lecture 33]{Tartar}. When deriving regularity results for the optimal variables of problem \eqref{eq:min}--\eqref{eq:cc}, it will be important to characterize the space $\Hs$ for $s\in(1,2]$. In fact, we have that, for such a range of values of $s$, $\Hs = H^s(\Omega)\cap H^1_0(\Omega)$; see \cite{ShinChan}. 

For $s\in(0,1)$ we denote by $\Hsd$ the dual of $\Hs$. With this notation, $\mathcal{L}^s : \Hs \to \Hsd$ is an isomorphism.

\subsection{Weighted Sobolev spaces}
The localization results by Caffarelli and Silvestre \cite{CS:07,CDDS:11,ST:10} require to deal with a nonuniformly elliptic equation posed on the semi--infinite cylinder $\C$. To analyze such an equation, it is instrumental to consider weighted Sobolev spaces with the weight $y^{\alpha}$ ($-1 < \alpha <1$ and $y \geq0$). We thus define
\begin{equation}
  \label{HL10}
  \HL(y^{\alpha},\C) = \left\{ w \in H^1(y^\alpha,\C): w = 0 \textrm{ on } \partial_L \C\right\}.
\end{equation}
For $\alpha \in (-1,1)$ we have that the weight $|y|^\alpha$ belongs to the so--called Muckenhoupt class $A_2(\R^{n+1})$, see \cite{Muckenhoupt,Turesson}. Consequently, $\HL(y^{\alpha},\C)$, endowed with the norm
\begin{equation}
 \label{eq:norm}
 \| w \|_{H^1(y^{\alpha},\C)} := \left( \| w \|_{L^2(y^{\alpha},\C)} + \| \nabla w \|_{L^2(y^{\alpha},\C)}\right)^{\frac{1}{2}}
\end{equation}
is a Hilbert space \cite[Proposition 2.1.2]{Turesson} and smooth functions are dense \cite[Corollary 2.1.6]{Turesson}; see also \cite[Theorem~1]{GU}. We recall the following \emph{weighted Poincar\'e inequality}: 
\begin{equation}
\label{Poincare_ineq}
\| w \|_{L^2(y^{\alpha},\C)} \lesssim \| \nabla w \|_{L^2(y^{\alpha},\C)}
\quad \forall w \in \HL(y^{\alpha},\C)
\end{equation}
\cite[ineq. (2.21)]{NOS}. We thus have that $\| \nabla w \|_{L^2(y^{\alpha},\C)}$ is equivalent to \eqref{eq:norm} in $\HL(y^{\alpha},\C)$. For $w \in H^1( y^{\alpha},\C)$, we denote by $\tr w$ its trace onto $\Omega \times \{ 0 \}$, and we recall (\cite[Prop.~2.5]{NOS})
\begin{equation}
\label{Trace_estimate}
\tr \HL(y^\alpha,\C) = \Hs,
\qquad
  \|\tr w\|_{\Hs} \lesssim  \| w \|_{\HLn(y^\alpha,\C)}.
\end{equation}

\subsection{Convex functions and subdifferentials}
Let $E$ be a real normed vector space. Let $\eta: E \rightarrow \R \cup \{\infty\}$ be convex and proper, and let $v \in E$ with $\eta(v) < \infty$. By convexity of $\eta$ and the fact that $\eta(v) < \infty$ we conclude that the graph of $\eta$ can always be minorized by a hyperplane. If $\eta$ is not differentiable at $v$, then a useful substitute for the derivative is a subgradient, which is nothing but the slope of a hyperplane that minorizes the graph of $\eta$ and is exact at $v$. In other words, a \emph{subgradient} of $\eta$ at $v$ is a continuous linear functional $v^*$ on $E$ that satisfies
\begin{equation}
 \label{eq:subgradient}
 \langle v^*,w - v \rangle \leq \eta(w) - \eta(v) \quad \forall w \in E,
\end{equation}
where $\langle \cdot, \cdot \rangle$  denotes the duality pairing between $E^{*}$ and $E$. We immediately remark that a function may admit many subgradients at a point of nondifferentiability. The set of all subgradients of $\eta$ at $v$ is called \emph{subdifferential} of $\eta$ at $v$ and is denoted by $\partial \eta (v)$. Moreover, by convexity, the subdifferential $\partial \eta(v) \neq \emptyset$ for all points $v$ in the interior of the effective domain of $\eta$. Finally, we mention that the subdifferential is monotone, \ie
\begin{equation}
\label{eq:subdiff_monotone}
  \langle v^* - w^*, v - w \rangle \geq 0 \quad \forall v^\star \in \partial \eta(v),\ \forall w^\star \in \partial\eta(w).
\end{equation}
We refer the reader to \cite{MR1058436,MR2330778} for a thorough discussion on convex analysis.

\section{The fractional optimal control problem}
\label{sec:fractional}
In this section we analyze the fractional optimal control problem \eqref{eq:min}--\eqref{eq:cc}. We derive existence and uniqueness results together with first order necessary and sufficient optimality conditions. In addition, in section \ref{subsec:regularity_fractional}, we derive regularity results for the optimal variables that will be essential for deriving error estimates for the scheme proposed in section \ref{sec:approximation}.

For $J$ defined as in \eqref{eq:min}, the fractional optimal control problem reads: Find $\min J(\usf,\zsf)$ subject to \eqref{eq:state_equation} and \eqref{eq:cc}. The set of \emph{admissible controls} is defined by
\begin{equation}
\Zad := \{\zsf \in L^2(\Omega): \asf \leq \zsf(x') \leq \bsf \quad \textrm{a.e.} \quad x' \in \Omega\},
\label{eq:Zad}
\end{equation}
which is a nonempty, bounded, closed, and convex subset of ·$L^2(\Omega)$. Since we are interested in the nondifferentiable scenario, we assume that $\asf$ and $\bsf$ are real constants that satisfy the property $\asf < 0 < \bsf$ \cite[Remark 2.1]{CHW:12}.  The desired state $\usfd \in L^2(\Omega)$ while $\sigma$ and $\nu$ are both real and positive parameters.

As it is customary in optimal control theory \cite{Lions,Tbook}, to analyze \eqref{eq:min}--\eqref{eq:cc}, we introduce the so--called control to state operator.

\begin{definition}[fractional control to state map]
The map $\mathbf{S}: L^2(\Omega) \ni \zsf \mapsto \usf(\zsf) \in \Hs$, where $\usf(\zsf)$ solves  \eqref{eq:state_equation}, is called the fractional control to state map.
\label{def:control_to_state}
\end{definition}

This operator is linear and bounded from $L^2(\Omega)$ into $\Hs$ \cite[Lemma 2.2]{CDDS:11}. In addition, since $\Hs \hookrightarrow L^2(\Omega)$, we may also consider $\mathbf{S}$ acting from $L^2(\Omega)$ into itself. With this operator at hand, we define the optimal fractional state--control pair.

\begin{definition}[optimal fractional state-control pair]
A state--control pair $(\bar \usf, \bar \zsf) \in \Hs \times \Zad$ is called optimal for \eqref{eq:min}--\eqref{eq:cc} if $\bar \usf = \mathbf{S} \bar \zsf$ and
\[
J(\bar \usf, \bar \zsf) \leq J(\usf, \zsf)
\]
for all $(\usf, \zsf) \in \Hs \times \Zad $ such that $\usf = \mathbf{S}\zsf$.
\end{definition}

With these elements at hand, we present an existence and uniqueness result.

\begin{theorem}[existence and uniqueness]
The fractional optimal control problem \eqref{eq:min}--\eqref{eq:cc} has a unique optimal solution $(\bar \usf,\bar \zsf) \in \Hs \times \Zad$.
\end{theorem}
\begin{proof}
Define the reduced cost functional
\begin{equation}
\label{eq:f}
f(\zsf) := J(\mathbf{S}\zsf,\zsf) =  \frac{1}{2} \| \mathbf{S}\zsf-\usfd\|_{L^2(\Omega)}^2 + \frac{\sigma}{2} \|\zsf \|_{L^2(\Omega)}^2 + \nu \|\zsf\|_{L^1(\Omega)}.
\end{equation}
In view of the fact that $\mathbf{S}$ is injective and continuous, it is immediate that $f$ is strictly convex and weakly lower semicontinuous. The fact that $\Zad$ is weakly sequentially compact allows us to conclude \cite[Theorem 2.14]{Tbook}.
\end{proof}

\subsection{First order optimality conditions}
\label{subsec:first_order_fractional}
The reduced cost functional $f$ is a proper strictly convex function. However, it contains the $L^1(\Omega)$--norm of the control variable and therefore it is not nondifferentiable at $0 \in L^2(\Omega)$. This leads to some difficulties in the analysis and discretization of \eqref{eq:min}--\eqref{eq:cc}, that can be overcome by using some elementary convex analysis \cite{MR1058436,MR2330778}. With this we shall obtain explicit optimality conditions for problem \eqref{eq:min}--\eqref{eq:cc}. We begin with the following  classical result; see, for instance, \cite[Chapter 4]{MR2330778}.

\begin{lemma}
Let $f$ be defined as in \eqref{eq:f}. The element $\bar \zsf \in \Zad$ is a minimizer of $f$ over $\Zad$ if and only if there exists a subgradient $\lambda^{\star} \in \partial f(\bar \zsf)$ such that
\[
 (\lambda^{\star}, \zsf - \bar \zsf)_{L^2(\Omega)} \geq 0
\]
for all $\zsf \in \Zad$.
\label{le:minimizer}
\end{lemma}

In order to explore the previous optimality condition, we introduce the following ingredients.

\begin{definition}[fractional adjoint state]
For a given control $\zsf \in \Zad$, the fractional adjoint state $\psf \in \Hs$, associated to $\zsf$, is defined as $\psf = \mathbf{S}^{\dagger}(\mathbf{S} \zsf - \usfd)$, where $\mathbf{S}^\dagger$ denotes the $L^2$--adjoint of $\mathbf{S}$.
\label{def:fractional_adjoint_state}
\end{definition}

We also define the convex and Lipschitz function $\psi: L^1(\Omega) \rightarrow \mathbb{R}$ by $\psi(\zsf) :=  \|  \zsf \|_{L^1(\Omega)}$ ---  the nondifferentiable component of the cost functional $f$ --- and 
\begin{equation}                                                                                                                                                                                                                                             \varphi:  L^2(\Omega) \rightarrow \mathbb{R}, \qquad \zsf \mapsto  \varphi(\zsf):=   \frac{1}{2} \|\mathbf{S}\zsf-\usfd\|_{L^2(\Omega)}^2 + \frac{\sigma}{2} \|\zsf \|_{L^2(\Omega)}^2,     
\label{eq:varphi}
\end{equation}
the differentiable component of $f$. Standard arguments yield that $\varphi$ is Fr\'echet differentiable with $\varphi'(\zsf) = \mathbf{S}^{\dagger} (\mathbf{S} \zsf - \usfd) + \sigma \zsf$ \cite[Theorem 2.20]{Tbook}.
Now, invoking Definition \ref{def:fractional_adjoint_state}, we obtain that, for $\zsf \in \Zad$, we have
\begin{equation}
\varphi'(\zsf) = \psf + \sigma \zsf.
\label{eq:varphi_prime}
\end{equation}
It is rather standard to see that $\lambda \in \partial \psi(\zsf)$ if and only if the relations
\begin{equation}
\lambda(x') =1, \; \zsf(x') > 0,  \qquad   \lambda(x') = - 1, \;  \zsf(x') < 0, \qquad  \lambda(x') \in [-1,1], \; \zsf(x') = 0
\label{eq:subdif_L1}
\end{equation}
hold for \textrm{a.e.} $x' \in \Omega$. With these ingredients at hand, we obtain the following necessary and sufficient optimality conditions for our optimal control problem; see also \cite[Theorem 3.1]{CHW:12} and \cite[Lemma 2.2]{MR2826983}.

\begin{theorem}[optimality conditions]
The pair $(\bar \usf, \bar \zsf) \in \Hs \times \Zad$ is optimal for problem \eqref{eq:min}--\eqref{eq:cc} if and only if $\bar \usf = \mathbf{S} \bar \zsf$ and $\bar \zsf$ satisfies the variational inequality
\begin{equation}
\label{eq:optimality_condition_fractional} 
\left( \bar \psf + \sigma \bar \zsf + \nu \bar \lambda, \zsf - \bar \zsf  \right)_{L^2(\Omega)} \geq 0 \quad \forall \zsf \in \Zad,
\end{equation}
where $\bar \psf = \mathbf{S}^{\dagger}(\mathbf{S} \bar \zsf - \usfd)$ and $\bar \lambda \in \partial \psi (\bar \zsf)$.
\label{th:optimality_condition_fractional}
\end{theorem}
\begin{proof}
Since the convex function $\varphi$ is Fr\'echet differentiable we immediately have that $\partial \varphi(\bar \zsf) = \varphi'(\bar\zsf)$ \cite[Proposition 4.1.8]{MR2330778}. We thus apply the sum rule \cite[Proposition 4.5.1]{MR2330778} to conclude, in view of the fact that $\psi$ is convex, that $\partial f (\bar \zsf) = \varphi'(\bar \zsf) + \nu \partial \psi(\bar \zsf)$. This, combined with Lemma \ref{le:minimizer} and \eqref{eq:varphi_prime} imply the desired variational inequality \eqref{eq:optimality_condition_fractional}.
\end{proof}

To present the following result we introduce, for $a,b\in \R$, the projection formula
\[
\Proj_{[a,b]} \wsf(x'):=\min\left\{b,\max\left\{a,\wsf(x')\right\}\right\}.
\]

\begin{corollary}[projection formulas]
Let $\bar \zsf$, $\bar \usf$, $\bar \psf$ and $\bar \lambda$ be as in Theorem \ref{th:optimality_condition_fractional}. Then, we have that 
\begin{align} 
  \bar \zsf(x') & = \Proj_{[\asf,\bsf]}\left( -\frac{1}{\sigma} \left( \bar \psf(x') + \nu \bar \lambda(x') \right) \right),
  \label{eq:projection_z}
  \\
  \bar \zsf(x') &= 0 \quad \Leftrightarrow \quad |\bar \psf (x')| \leq \nu,
  \label{eq:sparsity}
  \\
  \bar \lambda(x') & = \Proj_{[-1,1]}\left( -\frac{1}{\nu}  \bar \psf(x')  \right).
  \label{eq:projection_lambda}
\end{align}
\end{corollary}
\begin{proof}
See \cite[Corollary 3.2]{CHW:12}. 
\end{proof}

\begin{remark}[sparsity]
We comment that property \eqref{eq:sparsity} implies the sparsity of the optimal control $\bar \zsf$. We refer the reader to \cite[Section 2]{MR2556849} for a thorough discussion on this matter.
\end{remark}

\subsection{Regularity estimates}
\label{subsec:regularity_fractional}

Having obtained conditions that guarantee the existence and uniqueness for problem \eqref{eq:min}--\eqref{eq:cc}, we now study the regularity properties of its optimal variables. This is important since, as it is well known, smoothness and rate of approximation go hand in hand. Consequently, any rigorous study of an approximation scheme must be concerned with the regularity of the optimal variables. Here, on the the basis of a bootstraping argument inspired by \cite{MR3429730,MR3504977}, we obtain such regularity results.

\begin{theorem}[regularity results for $\bar \zsf$ and $\bar \lambda$]
 If $\usfd \in \mathbb{H}^{1-s}(\Omega)$, then the optimal control for problem \eqref{eq:min}--\eqref{eq:cc} satisfies that $\bar \zsf \in H_0^1(\Omega)$. In addition, the subgradient $\bar \lambda$, given by \eqref{eq:projection_lambda}, 
 satisfies that $\bar \lambda \in  H_0^1(\Omega)$.
 \label{th:regularity_z_lambda}
\end{theorem}
\begin{proof}
We begin the proof by invoking the convexity of $\Omega$, the fact that $\calLs$ is a pseudodifferential operator of order $2s$ and that $\bar \zsf \in \Zad \subset L^2(\Omega)$ to conclude that 
\begin{equation}
 \label{eq:basic_regularity}
 \bar \usf \in \mathbb{H}^{2s}(\Omega), \qquad  \bar \psf \in 
\mathbb{H}^{\kappa}(\Omega), \quad \kappa = \min \{4s, 1+s, 2 \};
\end{equation}
the space $\mathbb{H}^{\delta}(\Omega)$, for $\delta \in (0,2]$, was characterized in Section \ref{sub:fractional_powers}. We now consider the following cases:

\noindent \framebox{\emph{Case 1}, $s \in \left[\frac{1}{4},1\right)$:} We immediately obtain that $\bar \psf \in H_0^1(\Omega)$. This, in view of the projection formula \eqref{eq:projection_lambda} and \cite[Theorem A.1]{KSbook} implies that $\bar \lambda \in H_0^1(\Omega)$; notice that formula \eqref{eq:projection_lambda} preserves boundary values. Now, since both functions $\bar \psf$ and $\bar \lambda$ belong to $H_0^1(\Omega)$, an application, again, of \cite[Theorem A.1]{KSbook} and the projection formula \eqref{eq:projection_z}, for $\bar \zsf$, implies that $\bar \zsf \in H_0^1(\Omega)$. We remark that, in view of the assumption $\asf < 0 < \bsf$, the formula \eqref{eq:projection_z} also preserves boundary values.

\noindent \framebox{\emph{Case 2}, $s \in \left(0,\tfrac{1}{4}\right)$:} We now begin the bootstrapping argument like that in \cite[Lemma 3.5]{MR3429730}. In this case, \eqref{eq:basic_regularity} implies that $\bar \psf \in \mathbb{H}^{4s}(\Omega)$. This, on the basis of a nonlinear operator interpolation result as in \cite[Lemma 3.5]{MR3429730}, that follows from \cite[Lemma 28.1]{Tartar}, guarantees that $\bar \lambda \in \mathbb{H}^{4s}(\Omega)$. We notice, once again, that formula  \eqref{eq:projection_lambda} preserves boundary values. Similar arguments allow us to derive that $\bar \zsf \in \mathbb{H}^{4s}(\Omega)$.

\framebox{\emph{Case 2.1}, $s \in \left[\tfrac{1}{8},\tfrac{1}{4}\right)$:} Since $\bar \zsf \in \mathbb{H}^{4s}(\Omega)$, we conclude that $\bar \usf \in \mathbb{H}^{6s}(\Omega)$ and that $\bar \psf \in \mathbb{H}^{\varepsilon}(\Omega)$, where $\varepsilon = \min \{ 8s, 1+s\}$. We now invoke that $s \in \left[\tfrac{1}{8},\tfrac{1}{4}\right)$ to deduce that $\bar \psf \in H_0^{1}(\Omega)$. This, in view of \eqref{eq:projection_lambda}, implies that $\bar \lambda \in H_0^1(\Omega)$, which in turns, and as a consequence of \eqref{eq:projection_z}, allows us to derive that $\bar \zsf \in H_0^1(\Omega)$.

\framebox{\emph{Case 2.2},  $s \in \left(0,\tfrac{1}{8}\right)$:} As in Case 2.1 we have that $\bar \psf \in \mathbb{H}^{8s}(\Omega)$. We now invoke, again, a nonlinear operator interpolation argument to conclude that $\bar \lambda \in \mathbb{H}^{8s}(\Omega)$ and then that $\bar \zsf \in \mathbb{H}^{8s}(\Omega)$. These regularity results imply that $\bar \usf \in \mathbb{H}^{10s}(\Omega)$ and then that $\bar \psf \in \mathbb{H}^{\iota}(\Omega)$, where $\iota = \min \{ 12s, 1+s\}$. 

\indent \indent \framebox{\emph{Case 2.2.1}, $s \in \left(\tfrac{1}{12},\tfrac{1}{8}\right]$:} We immediately obtain that $\bar \psf \in H_0^1(\Omega)$. This implies that $\bar \lambda \in H_0^1(\Omega)$, and thus that $\bar \zsf \in H_0^1(\Omega)$.

\indent \indent \framebox{\emph{Case 2.2.2},  $s \in \left(0,\tfrac{1}{12}\right]$:} We proceed as before.

After a finite number of steps we can thus conclude that, for any $s \in (0,1)$, $\bar \lambda$ and $\bar \zsf$ belong to $H_0^1(\Omega)$. This concludes the proof.
\end{proof}

As a by-product of the proof of the previous theorem, we obtain the following regularity result for the optimal state and optimal adjoint state.

\begin{corollary}[regularity results for $\bar \usf$ and $\bar \psf$]
  If $\usfd \in \Ws$, then $\bar \usf \in \mathbb{H}^{l}(\Omega)$, where $l = \min \{ 1+ 2s,2 \}$ and $\bar \psf \in \mathbb{H}^\varpi(\Omega)$, where $\varpi = \min \{ 1+s, 2\}$.
\end{corollary}

\section{The extended optimal control problem}
\label{sec:extended}

In this section we invoke the localization results of Caffarelli and Silvestre \cite{CS:07} and their extensions \cite{CDDS:11,ST:10} to circumvent the nonlocality of the operator $\calLs$ in the state equation \eqref{eq:state_equation}. We follow \cite{MR3429730} and consider the \emph{equivalent extended optimal control problem:} Find $\min\{ J(\tr \ue,\zsf) : \ue \in \HL(y^{\alpha},\C), \zsf \in \Zad \}$ subject to the \emph{extended state equation}:
\begin{equation}
\label{eq:state_equation_extended}
\ue \in \HL(y^{\alpha},\C): \quad a(\ue,\phi)  = ( \zsf, \tr \phi )_{L^2(\Omega)} \quad \forall \phi \in \HL(y^{\alpha},\C),
\end{equation}
where, for all $w,\phi \in \HL(y^{\alpha},\C)$, the bilinear form $a$ is defined by
\begin{equation}
a(w,\phi) = \frac{1}{d_s} \int_{\C}  {y^{\alpha}\mathbf{A}(x',y)} \nabla w \cdot \nabla \phi + y^{\alpha} c(x') w \phi  \diff x .
\label{eq:a}
\end{equation}

To describe the optimality conditions we introduce the \emph{extended adjoint problem}: 
\begin{equation}
\label{eq:adjoint_equation_extended}
\pe \in \HL(y^{\alpha},\C): \quad a(\phi,\pe)  =  ( \tr \ue - \usfd , \tr \phi )_{L^2(\Omega)} \quad 
\forall \phi \in \HL(y^{\alpha},\C).
\end{equation}
The optimality conditions in this setting now read as follows: the pair $(\bar{\ue},\ozsf) \in \HL(y^{\alpha},\C) \times \Zad$ is optimal if and only if $\bar{\ue} = \ue(\ozsf)$ solves \eqref{eq:state_equation_extended} and
\begin{equation}
\label{op_extended}
  (\tr \bar{\pe} + \sigma \ozsf + \nu \bar \lambda, \zsf - \ozsf )_{L^2(\Omega)} \geq 0 \quad \forall \zsf \in \Zad,
\end{equation}
where $\bar\pe = \bar\pe(\ozsf) \in \HL(y^{\alpha},\C)$ solves \eqref{eq:adjoint_equation_extended} and $\bar \lambda \in \partial \psi(\bar \zsf)$ .

The results of Caffarelli and Silvestre \cite{CS:07,CDDS:11,ST:10} yield that $\tr \oue = \ousf$ and $\tr \ope = \opsf$, where $\bar \usf \in \Hs$ solves \eqref{eq:state_equation} and $\bar \psf \in \Hs$ is as in Definition \ref{def:fractional_adjoint_state}. This implies the equivalence of the fractional and extended optimal control problems; see also \cite[Theorem 3.12]{MR3429730}.

\section{The truncated optimal control problem}
\label{sec:truncated}
The state equation \eqref{eq:state_equation_extended} of the extended optimal control problem is posed on the infinite domain $\C$ and thus it cannot be directly approximated with finite element--like techniques. However, the result of Proposition \ref{pro:exponential_decay} below shows that the  optimal extended state $\bar \ue$ decays exponentially in the extended variable $y$. This suggests to truncate $\C$ to $\C_{\Y} = \Omega \times (0,\Y)$, for a suitable truncation parameter $\Y$, and seek solutions in this bounded domain.

\begin{proposition}[exponential decay]
\label{pro:exponential_decay}
For every $\Y \geq 1$, the optimal state $\oue = \oue(\ozsf) \in \HL(y^{\alpha},\C)$,
solution to problem \eqref{eq:state_equation_extended}, satisfies
\begin{equation}
\label{eq:exponential_decay}
  \|\nabla \oue \|_{L^2(y^{\alpha},\Omega \times (\Y,\infty))} \lesssim e^{-\sqrt{\lambda_1} \Y/2}
  \| \ozsf \|_{\Hsd},
\end{equation}
where $\lambda_1$ denotes the first eigenvalue of the operator $\mathcal{L}$.
\end{proposition}
\begin{proof}
See \cite[Proposition 3.1]{NOS}.
\end{proof}

This motivates the \emph{truncated optimal control problem}: Find
$
\min\{ J(\tr v,\rsf):v \in \HL(y^{\alpha},\C_{\Y}), \rsf \in \Zad \}
$
subject to the \emph{truncated state equation}: 
\begin{equation}
\label{eq:state_equation_truncated}
v \in \HL(y^{\alpha},\C_{\Y}): \quad  a_\Y(v,\phi) = ( \rsf, \tr \phi )_{L^2(\Omega)} \quad \forall \phi \in \HL(y^{\alpha},\C_\Y),
\end{equation}
where
\[
  \HL(y^{\alpha},\C_\Y) = \left\{ w \in H^1(y^\alpha,\C_\Y): w = 0 \text{ on }
    \partial_L \C_\Y \cup \Omega \times \{ \Y\} \right\},
\]
and for all $w,\phi \in \HL(y^{\alpha},\C_\Y)$, the bilinear form $a_{\Y}$ is defined by
\begin{equation}
a_\Y(w,\phi) = \frac{1}{d_s} \int_{\C_\Y} {y^{\alpha}\mathbf{A}(x',y)} \nabla w \cdot \nabla \phi + y^{\alpha} c(x') w \phi \diff x.
\label{eq:a_Y}
\end{equation}

To formulate optimality conditions we introduce the \emph{truncated adjoint problem}: 
\begin{equation}
\label{eq:adjoint_equation_truncated}
p \in \HL(y^{\alpha},\C_{\Y}): \quad a_\Y(\phi,p)  =  ( \tr v - \usfd , \tr \phi )_{L^2(\Omega)} \quad 
\forall \phi \in \HL(y^{\alpha},\C_{\Y}).
\end{equation}
With this adjoint problem at hand, we present necessary and sufficient optimality conditions for the truncated optimal control problem: the pair $(\bar v,\bar \rsf) \in \HL(y^{\alpha},\C_{\Y}) \times \Zad$ is optimal if and only if $\bar v = \bar v(\orsf)$ solves \eqref{eq:state_equation_truncated} and
\begin{equation}
\label{op_truncated}
  (\tr \bar{p} + \sigma \orsf + \nu \bar t, \rsf - \orsf )_{L^2(\Omega)} \geq 0 \quad \forall \rsf \in \Zad,
\end{equation}
where $\bar p = \bar p(\orsf) \in \HL(y^{\alpha},\C_{\Y})$ solves \eqref{eq:adjoint_equation_truncated} and $\bar t \in \partial \psi(\bar \rsf)$ .

We now introduce the following auxiliary problem: 
\begin{equation}
\label{eq:R}
\mathscr{R} \in \HL(y^{\alpha},\C): \quad a(\mathscr{R}, \phi) = (  \tr \bar{v} - \usfd, \tr \phi )_{L^2(\Omega)} \quad \forall \phi \in \HL(y^{\alpha},\C).
\end{equation}

The next result follows from \cite[Lemma 4.6]{MR3429730} and shows how $(\bar{v}(\orsf),\orsf)$ approximates $(\bar{\ue}(\ozsf),\ozsf)$.

\begin{theorem}[exponential convergence]
\label{thm:exp_convergence}
If  $(\bar{\ue}(\ozsf),\ozsf)$ and $(\bar{v}(\orsf),\orsf)$ are the optimal pairs for the extended and truncated optimal control problems, respectively, then
\begin{equation}
 \| \orsf - \ozsf \|_{L^2(\Omega)} \lesssim  e^{-\sqrt{\lambda_1} \Y/4} \left( \| \orsf \|_{L^2(\Omega)} + \| \usfd \|_{L^2(\Omega)} \right) ,
\label{eq:control_exp}
\end{equation}
and
\begin{equation}
\| \tr( \bar{\ue} - \bar{v}) \|_{L^2(\Omega)} 
\lesssim  e^{-\sqrt{\lambda_1} \Y/4} \left( \| \orsf \|_{L^2(\Omega)} + \| \usfd \|_{L^2(\Omega)} \right) .
\label{eq:state_exp}
\end{equation}
\end{theorem}
\begin{proof}
Set $\zsf = \orsf$ and  $\rsf = \ozsf$ in \eqref{op_extended} and \eqref{op_truncated}, respectively. Adding the obtained inequalities we arrive at the estimate
\[
\sigma \| \ozsf - \orsf  \| _{L^2(\Omega)}^2 \leq ( \tr (\bar{\pe} -  \bar{p}) + \nu(\bar \lambda - \bar t), \orsf - \ozsf )_{L^2(\Omega)}.
\]

As a first step to control the right hand side of the previous expression, we recall that $\bar \lambda \in \partial \| \bar \zsf \|_{L^1(\Omega)}$ and $\bar t \in \partial \| \bar \rsf \|_{L^1(\Omega)}$ so that, by \eqref{eq:subdiff_monotone},
\[
  \nu(\bar \lambda - \bar t, \orsf - \ozsf )_{L^2(\Omega)}  \leq 0.
\]
Consequently,
\begin{equation}
\label{eq:basic_estimate}
\sigma \| \ozsf - \orsf  \| _{L^2(\Omega)}^2 \leq ( \tr (\bar{\pe} -  \bar{p}), \orsf - \ozsf )_{L^2(\Omega)}.
\end{equation}

To control the right hand side of the previous expression, we add and subtract the adjoint state $\pe(\orsf)$ as follows:
\begin{equation*}
\sigma \| \ozsf - \orsf  \| _{L^2(\Omega)}^2 \leq ( \tr (\bar{\pe} - \pe(\orsf)), \orsf - \ozsf )_{L^2(\Omega)} + 
( \tr (\pe(\orsf) - \bar{p}), \orsf - \ozsf )_{L^2(\Omega)} = \textrm{I} + \textrm{II}.
\end{equation*}
Let us now bound $\textrm{I}$. Notice that $\bar{\pe} - \pe(\orsf) \in \HL(y^{\alpha},\C)$ solves
\begin{equation*}
a(\phi_{\pe},\bar{\pe} - \pe(\orsf))  =  ( \tr (\bar \ue - \ue(\orsf)) , \tr \phi_{\pe} )_{L^2(\Omega)} \quad 
\forall \phi_{\pe} \in \HL(y^{\alpha},\C).
\end{equation*}
On the other hand, we also observe that $\bar \ue - \ue(\orsf) \in  \HL(y^{\alpha},\C)$ solves
\begin{equation*}
a(\bar \ue -\ue(\orsf),\phi_{\ue})  = ( \ozsf - \orsf, \tr \phi_{\ue} )_{L^2(\Omega)} \qquad \forall \phi_{\ue} \in \HL(y^{\alpha},\C).
\end{equation*}
Setting $\phi_{\ue} = \bar{\pe} - \pe(\orsf)$ and $\phi_{\pe} =  \ue(\orsf) - \bar \ue$ we immediately  conclude that $\textrm{I} \leq 0$.

To control the term $\textrm{II}$  we write $\bar \pe(\orsf) -\bar{p}= (\bar \pe(\orsf) - \mathscr{R}) + (\mathscr{R} - \bar{p} )$, where $\mathscr{R}$ solves \eqref{eq:R}. The first term is controlled in view of the trace estimate \eqref{Trace_estimate}, the well--posedness of problem \eqref{eq:R} and an application of the estimate \cite[Theorem 3.5]{NOS}:
\begin{equation*}
\| \tr (\pe(\bar{\rsf}) - \mathscr{R} ) \|_{L^2(\Omega)} \lesssim \| \tr (\ue(\bar{\rsf}) - \bar v(\orsf) ) \|_{L^2(\Omega)} \lesssim e^{-\sqrt{\lambda_1} \Y/4} \| \orsf \|_{L^2(\Omega)}.
\end{equation*}
Similar arguments yield: $
\| \tr( \mathscr{R} - \bar p ) \|_{L^2(\Omega)} \lesssim e^{-\sqrt{\lambda_1} \Y/4} ( \| \orsf \|_{L^2(\Omega)} + \| \usfd \|_{L^2(\Omega)} ) . 
$
In view of \eqref{eq:basic_estimate}, a collection of these estimates allow us to obtain \eqref{eq:control_exp}.

The estimate \eqref{eq:state_exp} follows from similar arguments upon writing $
  \bar{\ue}- \bar{v}(\orsf) = \left( \bar{\ue}(\ozsf) - \ue(\orsf) \right) + \left( \ue(\orsf) - \bar{v}(\orsf) \right).
$
This concludes the proof.
\end{proof}

We now state projection formulas and regularity results for the optimal variables $\orsf$ and $\bar t$, together with a sparsity property for $\bar \rsf$.

\begin{corollary}[projection formulas]
Let the variables $\bar \rsf$, $\bar v$, $\bar p$ and $\bar t$ be as in the variational inequality \eqref{op_truncated}. Then, we have that 
\begin{align} 
  \bar \rsf(x') & = \Proj_{[\asf,\bsf]}\left( -\frac{1}{\sigma} \left( \tr \bar p(x') + \nu \bar t(x') \right) \right),
  \label{eq:projection_r}
  \\
  \bar \rsf(x') &= 0 \quad \Leftrightarrow \quad |\tr \bar p (x')| \leq \nu,
  \label{eq:sparsity_r}
  \\
  \bar t(x') & = \Proj_{[-1,1]}\left( -\frac{1}{\nu}  \tr \bar p(x')  \right).
  \label{eq:projection_t}
\end{align}
\end{corollary}
\begin{proof}
See \cite[Corollary 3.2]{CHW:12}. 
\end{proof}

\begin{proposition}[regularity results for $\bar \rsf$ and $\bar t$]
If $\usfd \in \mathbb{H}^{1-s}(\Omega)$, then the truncated optimal control $\bar \rsf \in H_0^1(\Omega)$. In addition, the subgradient $\bar t$, given by \eqref{eq:projection_t}, satisfies that $\bar t \in  H_0^1(\Omega)$.
\label{pro:regularity_r}
\end{proposition}
\begin{proof}
The proof is an adaption of the techniques elaborated in the proof of \cite[Proposition 4.1]{Otarola_controlp1} and the bootstrapping argument of Theorem \ref{th:regularity_z_lambda}.
\end{proof}

We conclude this section with regularity results for the traces of the optimal state and adjoint state.

\begin{corollary}[regularity results for $\tr \bar v$ and $\tr \bar p$]
  If $\usfd \in \Ws$, then $\tr \bar v \in \mathbb{H}^{l}(\Omega)$, where $l = \min \{ 1+ 2s,2 \}$ and $\tr \bar p \in \mathbb{H}^\varpi(\Omega)$, where $\varpi = \min \{ 1+s, 2\}$.
\label{pro:regularity_v}
\end{corollary}

\section{Approximation of the fractional control problem}
\label{sec:approximation}
In this section we design and analyze a numerical technique to approximate the solution of the optimal control problem \eqref{eq:min}--\eqref{eq:cc}. In order to make this contribution self--contained, we briefly review the finite element method proposed and developed for the state equation \eqref{eq:state_equation} in \cite{NOS}.

\subsection{A finite element method for the state equation}
\label{sec:approximation_NOS}
We follow \cite[Section 4]{NOS} and let $\T_{\Omega} = \{K\} $ be a conforming triangulation of $\Omega$ into cells $K$ (simplices or $n$--rectangles). We denote by $\Tr_{\Omega}$ the collection of all conforming refinements of an original mesh $\T_0$, and assume that the family $\Tr_{\Omega}$ is shape regular \cite{CiarletBook,Guermond-Ern}. If $\T_{\Omega} \in \Tr_{\Omega}$, we define $h_{\T_{\Omega}} = \max_{K\in \T_{\Omega}} h_K$. We construct a mesh $\T_{\Y}$ over $\C_\Y$ as the tensor product triangulation of $\T_{\Omega} \in \Tr_{\Omega}$ and $\I_{\Y}$, where the latter corresponds to a partition of the interval $[0,\Y]$ with mesh points:
\begin{equation}
y_k = \left( \frac{k}{M} \right)^{ \gamma}\Y, \quad k = 0,\cdots,M,
\label{eq:graded_mesh}
\end{equation}
with $\gamma = 3/(1-\alpha) = 3/(2s) > 1$. We notice that each discretization of the truncated cylinder $\C_\Y$ depends on the truncation parameter $\Y$. We denote by $\Tr$ the set of all such anisotropic triangulations $\T_\Y$. The following weak shape regularity condition is valid: there is a constant $\mu$ such that, for all $\T_\Y \in \Tr$, if $T_1 = K_1 \times I_1, T_2 = K_2 \times I_2 \in \T_\Y$ have nonempty intersection, then $h_{I_1} /h_{I_2} \leq \mu$, where $ h_I = | I |$ \cite{DL:05,NOS}. The main motivation for considering elements as in \eqref{eq:graded_mesh} is to compensate the rather singular behavior of $\ue$, solution to problem \eqref{eq:state_equation_extended}. We refer the reader to \cite{NOS} for details.

For $\T_{\Y} \in \Tr$, we define the finite element space 
\begin{equation}\
\label{eq:FESpace}
  \V(\T_\Y) = \left\{
            W \in C^0( \bar{\C}_\Y): W|_T \in \mathcal{P}_1(K) \otimes \mathbb{P}_1(I), \ \forall T \in \T_\Y, \
            W|_{\Gamma_D} = 0
          \right\},
\end{equation}
where $\Gamma_D = \partial_L \C_{\Y} \cup \Omega \times \{ \Y\}$ is the Dirichlet boundary. When the base $K$ of an element $T = K \times I$ is a simplex, the set $\mathcal{P}_1(K)$ is $\mathbb{P}_1(K)$. 
If $K$ is a cube, $\mathcal{P}_1(K)$ stands for $\mathbb{Q}_1(K)$. 
We also define $\U(\T)=\tr \V(\T_{\Y})$, \ie a $\mathcal{P}_1$ finite element space over the mesh $\T_{\Omega}$. Finally, we assume that every $\T_\Y \in \Tr$ is such that, $M \approx \# \T_{\Omega}^{1/n}$ so that, since $\#\T_{\Y} = M \, \# \T_{\Omega}$, we have $\#\T_\Y \approx M^{n+1}$.

The Galerkin approximation of \eqref{eq:state_equation_truncated} is defined as follows:
\begin{equation}
\label{eq:discrete_state}
 V \in \V(\T_{\Y}): \quad  a_{\Y}(V,W) = (\rsf, \textrm{tr}_{\Omega} W )_{L^2(\Omega)}
  \quad \forall W \in \V(\T_{\Y}),
\end{equation}
where $a_{\Y}$ is defined in \eqref{eq:a_Y}.  We present \cite[Theorem 5.4]{NOS} and \cite[Corollary 7.11]{NOS}.

\begin{theorem}[error estimates]
\label{TH:fl_error_estimates}
If $\ue(\rsf) \in \HL(y^{\alpha},\C)$ solves \eqref{eq:state_equation_extended} with $\zsf$ replaced by $\rsf \in \Ws$, then
\begin{equation}
\label{eq:fl_optimal_rate}
  \| \nabla( \ue(\rsf) - V) \|_{L^2(y^\alpha,\C)} \lesssim
|\log(\# \T_{\Y})|^s(\# \T_{\Y})^{-1/(n+1)} \|\rsf \|_{\mathbb{H}^{1-s}(\Omega)},
\end{equation}
provided $\Y \approx | \log(\# \T_{\Y}) |$. Alternatively, if $\usf(\rsf)$ denotes the solution to \eqref{eq:state_equation} with $\rsf$ as a forcing term, then
\begin{equation}
\label{eq:fl_rate}
\| \usf(\rsf) - \tr V \|_{\Hs} \lesssim
|\log(\# \T_{\Y})|^s(\# \T_{\Y})^{-1/(n+1)} \|\rsf \|_{\mathbb{H}^{1-s}(\Omega)}.
\end{equation}
\end{theorem}

\subsection{A fully discrete scheme for the fractional optimal control problem}
\label{sec:approximation_sparse}
In section \ref{sec:extended} we replaced the original fractional optimal control problem \eqref{eq:min}--\eqref{eq:cc} by an equivalent one that involves the local state equation \eqref{eq:state_equation_extended} and is posed on the semi--infinite cylinder $\C = \Omega \times (0,\infty)$. We then considered a truncated version of this, equivalent, control problem that is posed on the bounded cylinder $\C_{\Y} = \Omega \times (0,\Y)$ and showed that the error committed in the process is exponentially small. In light of these results, in this section we propose a fully discrete scheme to approximate the solution to \eqref{eq:min}--\eqref{eq:cc}: piecewise constant functions to approximate the control variable and, for the state variable, first--degree tensor product finite elements, as described in section \ref{sec:approximation_NOS}.

We begin by defining the set of discrete controls, and the discrete admissible set
\begin{align*}
  \mathbb{Z}(\T_\Omega) &= \left\{ Z \in L^{\infty}( \Omega ): Z|_K \in \mathbb{P}_0(K) \quad \forall K \in \T_\Omega \right\}, \\
  \mathbb{Z}_{ad}(\T_{\Omega}) &= \Zad \cap \mathbb{Z}(\T_\Omega),
\end{align*}
where $\Zad$ is defined in \eqref{eq:Zad}. Thus, the \emph{fully discrete optimal control problem} reads as follows: Find
$
 \min J(\tr V , Z) 
$
subject to the discrete state equation
\begin{equation}
\label{eq:state_equation_discrete}
 a_\Y(V,W) =  ( Z, \tr W )_{L^2(\Omega)} \quad \forall W \in \V(\T_{\Y}),
\end{equation}
and the discrete control constraints
$
Z \in \mathbb{Z}_{ad}(\T_{\Omega}).
$
We recall that the functional $J$ and the discrete space $\V(\T_{\Y})$ are defined by \eqref{eq:J} and \eqref{eq:FESpace}, respectively. 

We denote by $(\bar{V}, \bar Z) \in \V(\T_\Y) \times \mathbb{Z}_{ad}(\T_{\Omega})$ the optimal state--control pair solving the fully discrete optimal control problem; existence and uniqueness of such a pair being guaranteed by standard arguments. We thus define, in view of \cite{CS:07,NOS},
\begin{equation}
\label{eq:U_discrete}
 \bar{U}:= \tr \bar{V},
\end{equation}
to obtain a discrete approximation $(\bar{U},\bar Z) \in \U(\T_{\Omega}) \times \mathbb{Z}_{ad}(\T_{\Omega})$ of the optimal pair $(\ousf,\ozsf) \in \Hs \times \Zad$ that solves our original optimal control problem \eqref{eq:min}--\eqref{eq:cc}. We recall that $\U(\T_{\Omega}) = \tr \V(\T_\Y)$: a standard $\mathcal{P}_1$ finite element space over the mesh $\T_{\Omega}$.

\begin{remark}[locality] 
\rm 
The main advantage of the fully discrete optimal control problem is its \emph{local nature}: it involves the \emph{local} problem \eqref{eq:state_equation_discrete} as state equation.
\end{remark}

To present optimality conditions we define the optimal adjoint state:
\begin{equation}
\label{eq:adjoint_equation_discrete}
\bar{P} \in \V(\T_\Y): \quad a_\Y( W,\bar{P}) = ( \tr \bar{V} - \usfd , \textrm{tr}_{\Omega} W )_{L^2(\Omega)} \quad \forall W \in \V(\T_{\Y}).
\end{equation}

We provide first order necessary and sufficient optimality conditions for the fully discrete optimal control problem: the pair $(\bar V,\bar Z) \in \V(\T_{\Y}) \times \mathbb{Z}_{ad}(\T_{\Omega})$ is optimal if and only if $\bar V = \bar V(\bar Z)$ solves \eqref{eq:state_equation_discrete} and
\begin{equation}
\label{op_discrete}
  (\tr \bar{P} + \sigma \bar Z + \nu \bar \Lambda, Z - \bar{Z})_{L^2(\Omega)} \geq 0 \quad \forall Z \in \mathbb{Z}_{ad}(\T_{\Omega}),
\end{equation}
where $\bar P = \bar P( \bar Z ) \in \V(\T_{\Y})$ solves \eqref{eq:adjoint_equation_discrete} and $\bar \Lambda \in  \partial \psi(\bar Z)$.

We now explore the properties of the discrete optimal variables. By definition we have $\partial \psi(\bar Z) \subset \mathbb{Z}(\T_\Omega)^*$ and, consequently, $\bar \Lambda \in \psi(\bar Z)$ can be identified with an element of $\mathbb{Z}(\T_\Omega)$ that verifies
\begin{equation}
\label{eq:Lambda_is_P0}
 \bar \Lambda|_{K} = 1, \quad \bar Z|_{K} >0, \qquad  \bar \Lambda|_{K} = -1, \quad \bar Z|_{K} < 0,  \qquad \bar \Lambda|_{K} \in [-1,1], \quad \bar Z|_{K} = 0,
\end{equation}
for every $K \in \T_{\Omega}$. Consequently, by setting $Z = Z_{K} \in \mathbb{P}_0(K)$, that satisfies $\asf \leq Z_K \leq \bsf$, in \eqref{op_discrete} we arrive at
\[
 \sum_{K \in \T_{\Omega}} \left( \int_{K} \tr \bar P \diff x' + |K| \left( \sigma \bar Z|_{K} + \nu \bar \Lambda|_{K} \right) \right) \left( Z_K - \bar Z|_{K} \right) \geq 0.
\]
This discrete variational inequality implies the discrete projection formula
\begin{equation}
 \bar Z|_{K} = \Proj_{[\asf,\bsf]} \left( -\frac{1}{\sigma}\left[ \frac{1}{|K|} \int_{K} \tr \bar P \diff x' + \nu \bar \Lambda|_{K} \right] \right).
 \label{eq:projection_Z}
\end{equation}

On the basis of \eqref{eq:Lambda_is_P0} and \eqref{eq:projection_Z} we have that \cite[Section 4]{CHW:12}
\begin{equation*}
\bar Z|_{K} = 0 \quad \Leftrightarrow \quad  \frac{1}{|K|} \left| \int_{K} \tr \bar P \diff x' \right| \leq \nu \quad \forall K \in \T_{\Omega}
\label{eq:sparsity_Z}
\end{equation*}
and that
\begin{equation}
\bar \Lambda|_{K} = \Proj_{[-1,1]} \left(-\frac{1}{\nu |T|} \int_{K} \tr \bar P \diff x' \right) \quad \forall K \in \T_{\Omega}.
\label{eq:projection_Lambda}
\end{equation}

It will be useful, for the error analysis of the fully discrete optimal control problem, to introduce the $L^2$-orthogonal projection $\Pi_{\T_{\Omega}}$ onto $\mathbb{Z}(\T_{\Omega})$, which is defined as follows \cite{CiarletBook,Guermond-Ern}:
\begin{equation}
\label{eq:orthogonal_projection}
\Pi_{\T_{\Omega}}: L^2(\Omega) \rightarrow \mathbb{Z}(\T_{\Omega}), \qquad 
(\rsf - \Pi_{\T_{\Omega}}\rsf , Z ) = 0 \quad \forall Z \in  \mathbb{Z}(\T_{\Omega}).
\end{equation}
We recall the following properties of $\Pi_{\T_{\Omega}}$. 
\begin{enumerate}[1.]
 \item Stability: For all $\rsf \in L^2(\Omega)$, we have the bound
$
\|\Pi_{\T_{\Omega}}\rsf \|_{L^2(\Omega)} \lesssim
 \|\rsf \|_{L^2(\Omega)}.
$
\item Approximation property: If $\rsf \in H^{1}(\Omega)$, we have the error estimate
\begin{equation}
\label{o_p:aprox}
\| \rsf - \Pi_{\T_{\Omega}}\rsf \|_{L^2(\Omega)} \lesssim h_{\T_{\Omega}}\|\rsf \|_{H^1(\Omega)}
\end{equation}
where $h_{\T_{\Omega}}$ is defined as in Section \ref{sec:approximation_NOS}; see \cite[Lemma 1.131 and Proposition 1.134]{Guermond-Ern}.
\end{enumerate}

If $\rsf \in L^2(\Omega)$, \eqref{eq:orthogonal_projection} immediately yields 
$
 \Pi_{\T_{\Omega}}\rsf |_K = (1/|K|)\int_{K} \rsf \diff x'.
$
Consequently
\begin{equation}
\label{eq:subset}
\Pi_{\T_{\Omega}} \Zad \subset \mathbb{Z}_{ad}(\T_{\Omega}).
\end{equation}

We now introduce two auxiliary adjoint states. The first one is defined as the solution to: Find $Q \in \mathbb{V}(\T_{\Y})$ such that
\begin{equation}
\label{eq:aux1}
a_\Y( W,Q ) = ( \tr \bar{v}- \usfd , \textrm{tr}_{\Omega} W )_{L^2(\Omega)} \quad \forall W \in \mathbb{V}(\T_{\Y}).
\end{equation}
The second one solves: 
\begin{equation}
\label{eq:aux2}
R \in \mathbb{V}(\T_{\Y}): \quad a_\Y( W,R) = ( \tr V(\orsf) - \usfd , \textrm{tr}_{\Omega} W )_{L^2(\Omega)} \quad \forall W \in \mathbb{V}(\T_{\Y}),
\end{equation}
where $V(\orsf)$ corresponds to the solution to problem \eqref{eq:state_equation_discrete} with $Z$ replaced by $\bar \rsf$. 

With these ingredients at hand we now proceed to derive an a priori error analysis for the fully discrete optimal control problem.

\begin{theorem}[fully discrete scheme: error estimates]
\label{thm:error_estimates_1} 
Let $(\bar v, \orsf) \in \HL(y^{\alpha},\C_{\Y}) \times \Zad$ be the optimal pair for the truncated optimal control problem of section \ref{sec:truncated}, and let $(\bar V, \bar Z) \in \V(\T_{\Y}) \times \mathbb{Z}_{ad}(\T_{\Omega})$ be the solution to the fully discrete optimal control problem of section \ref{sec:approximation}. If $\usfd \in \mathbb{H}^{1-\epsilon}(\Omega)$, then
\begin{equation}
\label{eq:rsf-Z}
 \| \bar \rsf - \bar Z \|_{L^2(\Omega)} \lesssim | \log (\# \T_\Y) |^{2s} (\# \T_\Y)^{-\tfrac{1}{n+1}},
\end{equation}
and
\begin{equation}
\label{eq:errorstate}
 \| \tr (\bar v - \bar V)\|_{L^2(\Omega)} \lesssim | \log (\# \T_\Y) |^{2s} (\# \T_\Y)^{-\tfrac{1}{n+1}},
\end{equation}
where the hidden constants in both inequalities are independent of the discretization parameters but depend on the problem data.
\end{theorem}
\begin{proof}
We proceed in five steps.

\noindent \framebox{\emph{Step 1}.} We observe that since $\mathbb{Z}_{ad}(\T_{\Omega}) \subset \Zad$, we are allowed to set $\rsf = \bar Z$ in the variational inequality \eqref{op_truncated}. This yields the inequality
\[                                                                                                                                                                                  (\tr \bar{p} + \sigma \orsf + \nu \bar t, \bar Z - \orsf )_{L^2(\Omega)} \geq 0.                                                                                                                                                                                    \]
On the other hand, in view of \eqref{eq:subset}, we can set $Z = \Pi_{\T_{\Omega}} \orsf$ in \eqref{op_discrete} and conclude that
\[
 (\tr \bar{P} + \sigma \bar Z + \nu \bar \Lambda, \Pi_{\T_{\Omega}} \orsf- \bar Z)_{L^2(\Omega)} \geq 0.
\]
Since $\bar t \in \partial \psi(\orsf)$ and $\bar \lambda \in \partial \psi(\bar Z)$, \eqref{eq:subgradient} gives that the previous inequalities are equivalent to the following ones:
\begin{align}
\label{eq:basic_with_j_1}
 (\tr \bar{p} + \sigma \orsf , \bar Z - \orsf )_{L^2(\Omega)} + \nu (\psi(\bar Z) - \psi(\orsf)) &\geq 0,
 \\
 \label{eq:basic_with_j_2}
  (\tr \bar{P} + \sigma \bar Z ,  \Pi_{\T_{\Omega}} \orsf- \bar Z )_{L^2(\Omega)} + \nu (\psi( \Pi_{\T_{\Omega}} \orsf ) - \psi(\bar Z))  &\geq 0.
\end{align}
We recall that $\psi(\wsf) = \|\wsf \|_{L^1(\Omega)}$. Invoking the fact that $\Pi_{\T_{\Omega}}$ is defined as in \eqref{eq:orthogonal_projection}, we conclude that $\psi( \Pi_{\T_{\Omega}} \orsf ) \leq \psi(\orsf)$, and thus $ (\psi(\bar Z) - \psi(\orsf)) + (\psi( \Pi_{\T_{\Omega}} \orsf ) - \psi(\bar Z)) \leq 0$. The latter and the addition of the inequalities \eqref{eq:basic_with_j_1} and \eqref{eq:basic_with_j_2} imply that
\[
 (\tr \bar{p} + \sigma \orsf , \bar Z - \orsf )_{L^2(\Omega)} + (\tr \bar{P} + \sigma \bar Z ,  \Pi_{\T_{\Omega}} \orsf- \bar Z )_{L^2(\Omega)} \geq 0,
\]
which yields the basic error estimate
\begin{equation}
 \label{eq:basic_estimate_discrete}
 \begin{aligned}
   \sigma \| \orsf - \bar Z \|^2_{L^2(\Omega)} &\leq (\tr(\bar p - \bar P),  \bar Z - \bar \rsf   )_{L^2(\Omega)} + (\tr \bar P + \sigma \bar Z,  \Pi_{\T_{\Omega}} \orsf- \bar \rsf  )_{L^2(\Omega)} \\
   &= \mathrm{I} + \mathrm{II}.
 \end{aligned}
\end{equation}
 
\noindent \framebox{\emph{Step 2}.} The goal of this step is to control the term $\mathrm{I}$ in \eqref{eq:basic_estimate_discrete}. To do this, we use the auxiliary adjoint states $Q$ and $R$ defined by problems \eqref{eq:aux1} and \eqref{eq:aux2}, respectively, and write
  \begin{equation}
  \label{eq:I1_I2_I3}
  \begin{aligned}
  \mathrm{I} &=  (\tr(\bar p -  Q),  \bar Z - \bar \rsf  )_{L^2(\Omega)}  +  (\tr( Q - R),  \bar Z - \bar \rsf   )_{L^2(\Omega)} \\
  &+  (\tr(R - \bar P),  \bar Z - \bar \rsf  )_{L^2(\Omega)} \\ &=: \mathrm{I}_1 +  \mathrm{I}_2 +  \mathrm{I}_3.
  \end{aligned}
  \end{equation}
  
To bound the term $\textrm{I}_1$ we realize that $Q$, defined as the solution to \eqref{eq:aux1}, is nothing but the Galerkin approximation of the optimal adjoint state $\bar p$. Consequently, an application of the error estimate of \cite[Proposition 28]{NOS3} yields 
\begin{equation}
\label{eq:III_2_1}
\| \tr(\bar p - \bar Q) \|_{L^2(\Omega)} \lesssim |\log N|^{2s} N^{-\frac{1+s}{n+1}} \left( \|\tr \bar v \|_{\Ws} + \| \usfd\|_{\Ws} \right),
\end{equation}
where $N = \# \T_{\Y}$. We note that the $\Ws$--norm of $\tr \bar{v}$ is uniformly controlled in view of Corollary \ref{pro:regularity_v}.

We now bound the term $\textrm{I}_2$. To accomplish this task, we invoke the trace estimate \eqref{Trace_estimate}, a stability estimate for the discrete problem that $Q-R$ solves and the error estimate of  \cite[Proposition 28]{NOS3}. In fact, these arguments allow us to obtain 
\begin{equation}
\label{eq:III_2_2}
  \begin{aligned}
    \| \tr(Q - R) \|_{L^2(\Omega)} & \lesssim \| \nabla ( Q - R)) \|_{L^2(y^{\alpha},\C_{\Y})} \lesssim \| \tr(\bar v - V(\orsf)) \|_{\Hsd}
    \\
    & \lesssim \| \tr (\bar v - V(\orsf)) \|_{L^2(\Omega)}
  \lesssim  |\log N|^{2s} N^{-\frac{1+s}{n+1}}  \| \orsf \|_{\Ws}.
  \end{aligned}
\end{equation}
We remark that, in view of the results of Proposition \ref{pro:regularity_r}, we have that $\orsf \in H_0^1(\Omega) \hookrightarrow \Ws$ for $s \in (0,1)$.

We now estimate the remaining term $\textrm{I}_3$. To do this, we set $W = V(\orsf) - \bar V \in \V(\T_{\Y})$ as a test function in the problem that $R - \bar P$ solves. This yields
\[
 a_{\Y} (V(\orsf) - \bar V,R - \bar P) = ( \tr(V(\orsf) - \bar V), \tr(V(\orsf) - \bar V) )_{L^2(\Omega)}.
\]
Similarly, by setting $W = R - \bar P \in \V(\T_{\Y})$ as a test function in the problem that $V(\orsf) - \bar V$ solves we arrive at
\[
 a_{\Y} (V(\orsf) - \bar V, R - \bar P) =(\orsf - \bar Z, \tr ( R - \bar P))_{L^2(\Omega)}.
\]
Consequently,
\[
 \mathrm{I}_3 = (\tr(R - \bar P),  \bar Z - \bar r  )_{L^2(\Omega)} = - \| \tr( V(\orsf) - \bar V )\|_{L^2(\Omega)}^{2} \leq 0.
\]

\noindent \framebox{\emph{Step 3.}}  In this step we bound the term $\textrm{II}=(\tr \bar P + \sigma \bar Z,  \Pi_{\T_{\Omega}} \orsf- \bar \rsf  )_{L^2(\Omega)}$ in \eqref{eq:basic_estimate_discrete}. We begin by rewriting $\textrm{II}$ as follows:
\begin{multline*}
 \textrm{II} = (\tr \bar p + \sigma \bar \rsf,  \Pi_{\T_{\Omega}} \orsf - \bar \rsf )_{L^2(\Omega)} + (\tr (\bar P \pm R \pm Q -\bar p),  \Pi_{\T_{\Omega}} \orsf - \bar \rsf )_{L^2(\Omega)}
 \\
 +\sigma (\bar Z - \bar \rsf, \Pi_{\T_{\Omega}} \orsf - \bar \rsf )_{L^2(\Omega)} = \textrm{II}_1 + \textrm{II}_2 + \textrm{II}_3.  
\end{multline*}
The control of the first term, $\textrm{II}_1$ follows from the definition \eqref{eq:orthogonal_projection} of $\Pi_{\T_{\Omega}}$, its approximation property \eqref{o_p:aprox} and the regularity results of Propositions \ref{pro:regularity_r} and \ref{pro:regularity_v}:
\begin{align*}
 \textrm{II}_1 &= (\tr \bar p + \sigma \bar \rsf  -\Pi_{\T_{\Omega}}( \tr \bar p + \sigma \bar \rsf  ) ,  \Pi_{\T_{\Omega}} \orsf - \bar \rsf )_{L^2(\Omega)}
 \\
 & \lesssim h^2_{\T_{\Omega}} \| \tr \bar p + \sigma \bar \rsf\|_{H^1(\Omega)} \| \orsf \|_{H^1(\Omega)}.
\end{align*}
The term $\textrm{II}_2$ is bounded by employing the arguments of Step 3: $\tr(\bar P - R)$ is controlled in view of the trace estimate \eqref{Trace_estimate} and the stability of the problems that $\bar{P} - R$ and $V(\orsf) - \bar V$ solve:
\[
 \| \tr( \bar{P} - R) \|_{L^2(\Omega)} \lesssim \| \tr (\bar V-  V(\orsf))\|_{\Hsd} \lesssim \|  \bar Z - \bar \rsf\|_{L^2(\Omega)}.
\]
The terms $\tr ( R - Q)$ and  $\tr ( Q - \bar p)$ are bounded as in \eqref{eq:III_2_2} and \eqref{eq:III_2_1}, respectively. The estimate for $\textrm{II}_3$ is a trivial consequence of the Cauchy--Schwarz inequality. 

\noindent \framebox{\emph{Step 4.}} The desired error bound \eqref{eq:rsf-Z} follows from collecting all estimates that we obtained in previous steps and recalling that $h_{\T_{\Omega}} \approx (\# \T_\Y)^{-1/(n+1)}$.

\noindent \framebox{\emph{Step 5.}} We finally derive estimate \eqref{eq:errorstate}. A basic application of the triangle inequality yields
\[
\| \tr (\bar{v} -  \bar{V})\|_{\Hs} \leq  \| \tr (\bar v - V(\orsf) ) \|_{\Hs}
+ \|\tr(  V(\orsf)  -  \bar V) \|_{\Hs}.
\]
The estimate for the term $ \| \tr( \bar v - V(\orsf)) \|_{\Hs}$ follows by applying the error estimate \eqref{eq:fl_rate}. To control the remaining term $ \|  V(\orsf)  -  \bar V \|_{\Hs}$ we invoke a stability result and estimate \eqref{eq:rsf-Z}. A collection of these estimates yields \eqref{eq:errorstate}. This concludes the proof.
\end{proof} 

As a consequence of the estimates of Theorems \ref{thm:exp_convergence} and \ref{thm:error_estimates_1} we arrive at the completion of the a priori error analysis for the fully discrete optimal control problem.

\begin{theorem}[fractional control problem: error estimates]
Let $(\bar{V},\bar{Z}) $ $\in \V(\T_\Y) \times \mathbb{Z}_{ad}(\T_{\Omega})$ be the optimal pair for the fully discrete optimal control problem of section \ref{sec:approximation} and let $\bar{U} \in \U(\T_{\Omega})$ be defined as in \eqref{eq:U_discrete}. If $\usfd \in \mathbb{H}^{1-s}(\Omega)$, then
\begin{equation}
\label{eq:final_estimate_control}
  \| \ozsf - \bar{Z} \|_{L^2(\Omega)} \lesssim  |\log (\# \T_{\Y})|^{2s}(\# \T_{\Y})^{-\tfrac{1}{n+1}},
\end{equation}
and
\begin{equation}
\label{eq:final_estimate_state}
\| \ousf - \bar{U} \|_{\Hs} \lesssim  |\log (\# \T_{\Y})|^{2s}(\# \T_{\Y})^{-\tfrac{1}{n+1}},
\end{equation}
where the hidden constants in both inequalities are independent of the discretization parameters but depend on the problem data.
\end{theorem}
\begin{proof}
To obtain the error estimate \eqref{eq:final_estimate_control} we invoke the estimates \eqref{eq:control_exp} and \eqref{eq:rsf-Z}. In fact, we have that
\begin{align*}
\| \ozsf - \bar{Z} \|_{L^2(\Omega)} & \leq 
 \| \ozsf - \orsf  \|_{L^2(\Omega)} +  \| \orsf - \bar{Z}\|_{L^2(\Omega)} 
\\
& \lesssim  e^{-\sqrt{\lambda_1} \Y/4} 
+  |\log(\# \T_{\Y})|^{2s}(\# \T_{\Y})^{-\tfrac{1}{n+1}}.
\end{align*}
The election of the truncation parameter $\Y \approx | \log(\# (\T_\Y)) |$ allows us to conclude; see \cite[Remark 5.5]{NOS} for details. Finally, to derive \eqref{eq:final_estimate_state}, we use that $\bar \usf = \tr \bar \ue$,  $\bar U = \tr \bar V$ and apply the estimates \eqref{eq:state_exp} and \eqref{eq:errorstate} as follows:
\begin{align*}
\| \ousf - \bar{U} \|_{\Hs} 
 & \leq 
\| \ousf - \tr \bar{v}  \|_{\Hs} +  \| \tr \bar{v} - \bar{U} \|_{\Hs} 
\\
&\lesssim e^{-\sqrt{\lambda_1} \Y/4} 
+  |\log(\# \T_{\Y})|^{2s}(\# \T_{\Y})^{-\tfrac{1}{n+1}}.
\end{align*}  
The fact that $\Y \approx | \log(\# (\T_\Y)) |$ yields \eqref{eq:final_estimate_state} and concludes the proof.
\end{proof}

\bibliographystyle{siamplain}
\bibliography{biblio}
\end{document}